\documentclass{article}
\usepackage{times}
\usepackage{amssymb, amsmath, amsthm}
\usepackage{float, multicol}
\usepackage{caption}
\usepackage{hyperref}
\usepackage[T1]{fontenc}

\restylefloat{table}

\usepackage[margin=1.0in]{geometry}
\usepackage{booktabs}

\newtheorem{theorem}{Theorem}[section]
\newtheorem{lemma}[theorem]{Lemma}
\newtheorem{proposition}[theorem]{Proposition}

\theoremstyle{definition}

\newtheorem{algorithm}[theorem]{Algorithm}

\theoremstyle{remark}
\newtheorem{remark}[theorem]{Remark}

\newcommand{\F}{\mathbb{F}}
\newcommand{\Z}{\mathbb{Z}}
\newcommand{\C}{\mathbb{C}}
\newcommand{\Cr}{\mathrm{C}}

\newcommand{\CGW}{\mathrm{CGW}}
\newcommand{\BH}{\mathrm{BH}}
\newcommand{\We}{\mathrm{W}}

\newcommand{\GF}{\mathrm{GF}}

\newcommand{\SAut}{\mathrm{SAut}}
\newcommand{\Mon}{\mathrm{Mon}}
\newcommand{\diag}{\mathrm{diag}}
\newcommand{\CT}{\mathrm{CT}}

\newcommand\myH{{\bf H}}
\newcommand\myT{{\bf T}}
\newcommand\myU{\chi_{{\bf H}}}

\newcommand{\myuparrow}{\!\uparrow}

\title{Complex generalised weighing matrices in centraliser algebras of monomial representations} 

\author{
Ronan Egan\footnotemark[2],~
Padraig \'{O} Cath\'{a}in\footnotemark[3]~~
and Andrea {\v S}vob\footnotemark[1]}
\date{
\today
}

\begin{document}

\maketitle

\begin{abstract}
An $n \times n$ matrix $W$ with exactly $w$ non-zero entries taken from the set of $k^{\rm th}$ complex roots of unity in each row and column satisfying $WW^{\ast} = wI_n$ is a complex generalised weighing matrix $\CGW(n,w;k)$.
We study such matrices through the centraliser algebras of monomial representations of finite groups. Using an exhaustive search over the linear characters of Schur covers, we classify, up to monomial equivalence, the complex generalised weighing matrices admitting a primitive group of rank at most five and degree at most $80$ acting by strong automorphisms, for coefficient orders $k \leq 6$, with partial results for larger degrees $100$. The census recovers known infinite families related to projective and affine finite geometries, describes infinite families related to Hamming schemes and settles the existence of some small open cases enumerated in the literature. We construct quantum error-correcting codes from the these matrices and determine their minimum distances exactly in all cases.
\end{abstract}

\renewcommand{\thefootnote}{\fnsymbol{footnote}}
\footnotetext[2]{School of Mathematical Sciences, Dublin City University, Ireland. Email: \href{mailto:ronan.egan@dcu.ie}{ronan.egan@dcu.ie}}
\footnotetext[3]{\'{U}dar\'{a}s na Gaeltachta, Ireland. Email: \href{mailto:p.ocathain@gmail.com}{p.ocathain@gmail.com}}
\footnotetext[1]{Faculty of Mathematics, University of Rijeka, Croatia. Email: 
\href{mailto:asvob@math.uniri.hr}{asvob@math.uniri.hr}}
\renewcommand{\thefootnote}{\roman{footnote}}



\section{Introduction}

A recurring theme in the study of Hadamard-type matrices is the choice of
a measure of complexity.  Historically, the natural measures were the
order of the matrix and the size of its set of entries; matrices with
entries among the $k$-th roots of unity have accordingly been studied
extensively see~\cite{Karol} and the references contained therein for extensive and up to date literature on the subject. In~\cite{Monrep},
the second and third authors, together with Barrera Acevedo and Dietrich,
proposed an alternative measure: the dimension of an associative algebra
containing the matrix. In particular, a matrix invariant under a group of automorphisms of small rank necessarily lies in the centraliser algebra of that linear group, 
which can be studied using methods of representation theory. If the algebra is commutative, 
matrix eigenvalues are expressed as character sums over double cosets, so that conditions on entries and eigenvalues -- which characterise Hadamard and weighing matrices alike --
reduce to a finite system of norm equations.  The present paper extends that programme to weighing matrices: we classify the complex generalised weighing
matrices lying in centraliser algebras of monomial groups of dimension at most 6 by exhaustive computer search, and as a consequence we exhibit the some new families and isolated examples of highly symmetric complex generalized weighing matrices. 

To be precise, a \textit{complex generalised weighing matrix} of order~$n$, weight~$w$,
and phase~$k$ is a square matrix with entries in the $k$-th roots of unity
and zero satisfying $WW^{\ast} = wI_n$; the set of all such matrices is
denoted $\CGW(n,w;k)$.  When $w = n$ one recovers the Butson Hadamard
matrices $\BH(n,k)$, and when $k = 2$ the real weighing matrices
$\We(n,w)$.  The natural equivalence operations permute and rescale rows
and columns by roots of unity, so the automorphism group of $W$ is isomorphic to a subgroup of the monomial group $\Mon_n(\langle\zeta_k\rangle)$, and a
matrix $W$ invariant under a large monomial group, i.e., $W = PWP^{\ast}$ for all $P \in G \leq \Mon_n(\langle\zeta_k\rangle)$, lies in a centraliser
algebra of small dimension.  Studying $W$ through this centraliser algebra
is the organising principle of the paper, following the research project mentioned previously and an independent investigation of Goldberger and collaborators, \cite{Monrep,
Goldberger2024}.

For each primitive permutation group $G$
of rank $r \in \{3,4,5\}$ and degree $n \leq 80$, we enumerate the $\CGW$
matrices in the centraliser algebras of the monomial representations of a
Schur cover of $G$, with coefficients drawn from
$\{0\} \cup \langle \zeta_k \rangle$ for $k \in \{2,3,4,5,6\}$.
Weighing matrices are relatively abundant: we find over 3,000 weighing matrices 
up to degree 80, in 136 equivalence classes. Additional searches at higher degrees 
and over almost simple sporadic groups and groups of Lie type gave additional 
results, including the smallest members of infinite families described previously by 
Seberry and Whiteman, by Berman and by Goldberger, \cite{SeberryWhiteman1975, Berman1978, Goldberger2024}. For 2-transitive groups, our results corroborate those of Moorhouse, \cite{Moorhouse2001}.

\subsection{Structure of paper}

Section~\ref{sec:prelim} recalls the centraliser-algebra machinery
of~\cite{Monrep} and the canonical-form methods used to verify and
organise the census. Two ingredients are new: an adaptation of the
structure theory of strong automorphism groups to matrices with zero
entries, and Proposition~\ref{prop:lifting}, which uses Schur's theory of
projective representations to show that the search over linear characters
of a single Schur cover is exhaustive, without assuming that the groups involved are perfect.

Section~\ref{sec:families} presents the classification; its precise scope
is stated as Theorem~\ref{thm:scope}. Apart from the search framework that we develop, our main theoretical contribution to the classification of generalised weighing matrices are theorems classifying weighing matrices in certain monomial covers of Hamming schemes, given in Theorems \ref{thm:hamming-flag} and \ref{thm:two-relation}. 

Among the isolated matrices found are a $\CGW(15,7;3)$ and a $\CGW(10,7;6)$, resolving two open cases in Egan's existence tables, \cite{CGW-Survey}. We also found a symmetric $\We(21,9)$, listed as an open case in the Handbook of Combinatorial
Designs, and discovered independently in 2025, \cite{Handbook2007, Rosin2025}. Finally, we constructed a skew-symmetric signing $S$ of the Hall--Janko graph with $SS^{\top} = 36I_{100}$ from the double cover of the Hall-Janko group. 

Section~\ref{sec:quantum} applies the newly located matrices to the construction of Hermitian self-orthogonal codes over fields of square order, and hence quantum error-correcting codes.  We determine the exact minimum distance of every code constructed (Table~\ref{tab:qcodes}); we identify an infinite family of $[[4^d,\, 4^d - 2^{d+1},\, 3]]$ codes supported on the Hamming-scheme matrices (Proposition~\ref{prop:family-distance}); and we obtain a $[[25,17,3]]_2$ code matching the best known parameters at its length and dimension.  Section~\ref{sec:discussion} concludes with a comparison against the published catalogues and two open problems.

\section{Preliminaries}\label{sec:prelim}

Throughout this paper, let $\zeta_{k} = e^{\frac{2\pi\sqrt{-1}}{k}}$, and denote by $\Mon_{n}(\langle \zeta_{k} \rangle)$ the group of monomial $n \times n$ matrices with entries in the $k^{\textrm{th}}$ roots of unity. There is a canonical projection $\Mon_{n}(\langle \zeta_{k} \rangle) \rightarrow \Mon_n(\langle 1\rangle)$ onto permutation matrices, with diagonal matrices in the kernel. We use the terminology of permutation groups to describe monomial groups, discussing for example orbits, stabilisers, and transitive actions.

\subsection{Complex generalised weighing matrices}

A \textit{complex generalised weighing matrix} of \textit{order}~$n$, \textit{weight}~$w$, and \textit{phase}~$k$ is an $n \times n$ matrix $W$ with entries in $\{0\} \cup \langle \zeta_k \rangle$ satisfying $WW^{\ast} = wI_n$, where $W^{\ast}$ denotes the conjugate transpose. The set of all such matrices is denoted $\CGW(n, w; k)$. When $w = n$ one recovers the \textit{Butson Hadamard matrices} $\BH(n, k) = \CGW(n, n; k)$, and when $k = 2$ the real \textit{weighing matrices} $\We(n, w) = \CGW(n, w; 2)$. An \textit{automorphism} of $W \in \CGW(n,w;k)$ is a pair of monomial matrices $(P, Q) \in \Mon_n(\langle \zeta_k \rangle)^2$ satisfying $PWQ^{\ast} = W$. A \textit{strong automorphism} is an automorphism of the form $(P,P)$; the group of all strong automorphisms is denoted $\SAut(W) \leq \Mon_n(\langle \zeta_k \rangle)^2$.

\subsection{Equivalence for generalized weighing matrices}\label{sec:equiv}

Two matrices $W_1, W_2 \in \CGW(n,w;k)$ are \textit{classically equivalent} if $W_2 = PW_1Q^{\ast}$ for some $P, Q \in \Mon_n(\langle\zeta_k\rangle)$. \textit{Extended equivalence} additionally allows field automorphisms $\zeta_k^e \mapsto \zeta_k^{re}$ ($\gcd(r,k)=1$) and transposition $W \mapsto W^\top$;
this matches the convention used by de Launey and Flannery, \cite{DeLauneyFlannery}. 

We reduce classical equivalence testing to directed graph isomorphism as follows. Define a directed graph $\mathcal{B}(W)$ with $2nk$ vertices partitioned into $n$ \textit{row blocks} $R_0,\ldots,R_{n-1}$ and $n$ \textit{column blocks} $C_0,\ldots,C_{n-1}$, each of size~$k$.  For each nonzero entry $W[i][j]=\zeta_k^e$, add the arc $(R_i[b],\, C_j[(b+e)\bmod k])$ for every $b\in\Z/k\Z$.  In addition,
add the directed $k$-cycle $R_i[0]\to R_i[1]\to\cdots\to R_i[k-1]\to R_i[0]$ within
each row block, and similarly within each column block. These directed $k$-cycles ensure that a potential automorphism maps row blocks to row blocks and column blocks to column blocks, and that the action on a particular row block corresponds to multiplication by a root of unity.
Consequently $W\mapsto\mathcal{B}(W)$ induces a bijection between classical
equivalence classes of $\CGW(n,w;k)$ matrices and isomorphism classes of
such digraphs. 

Canonical forms are computed via \textit{nauty}~\cite{McKay2014} in \textsc{SageMath}. The \textit{extended canonical key} of $W$ is the lexicographically smallest classical canonical label over the orbit
$\{\varphi_r(W),\,\varphi_r(W^\top): r\in(\Z/k\Z)^\times\}$ where  $\varphi_{r}$ is a field automorphism applied entrywise. Two matrices are extended-equivalent if and only if their extended canonical keys coincide.

\subsection{Monomial representations and their centraliser algebras}

This section summarises the machinery of \cite[Sections~3 and~5]{Monrep}
in the form we use; all statements needed in the sequel are reproduced in
full.  Let $G$ be a finite group with subgroup $H \leq G$ of index $n$,
let $T = \{t_1, \ldots, t_n\}$ with $t_1 = 1$ be a right transversal of
$H$ in $G$, so that every $g \in G$ factors uniquely as $g = h_g t_g$ with
$h_g \in H$ and $t_g \in T$, and define $\myH(g) = h_g$ and
$\myT(g) = t_g$.  Let $\chi \colon H \rightarrow \mathbb{C}^{\times}$ be a
linear character, extended to $\chi^{+} \colon G \rightarrow \mathbb{C}$
by zero outside $H$, and write $\myU(g) = \chi(\myH(g))$.  The group $G$
acts on $T$ by $t_i \cdot g = \myT(t_i g)$; orbits of the diagonal action
on $T \times T$ are called \textit{orbitals}, and their number is the
\textit{rank} $r$ of the permutation group.  The \textit{centraliser
algebra} $\Cr(\rho)$ of a representation $\rho$ is the algebra of all
matrices commuting with $\rho(G)$.  Matrices below have rows and columns
labelled by $T$, with entries $m(t_i, t_j)$.

\begin{proposition}[\cite{Monrep}, Prop.~3.1]\label{prop:monrep}
Every transitive $n$-dimensional monomial representation of $G$ is induced from a linear character $\chi$ of a point stabiliser $H \leq G$. The induced representation $\rho_\chi = \chi \myuparrow_H^G$ maps $g \in G$ to the matrix
\[
\rho_\chi(g) = \bigl[\chi^{+}(t_i g t_k^{-1})\bigr]_{i,k}.
\]
\end{proposition}

A matrix $M$ lies in $\Cr(\rho_\chi)$ if and only if
$m(\myT(g), \myT(tg)) = m(1, t)\,\myU(g)^{-1}\myU(tg)$ for all $g \in G$
and $t \in T$ \cite[Prop.~3.2]{Monrep}: the entries of $M$ along an
orbital are determined, via explicit character values, by the entry at
its representative $(1,t)$.  The scalar $\myU(g)^{-1}\myU(tg)$ may,
however, depend on $g$ and not only on the pair
$(\myT(g), \myT(tg))$; when it does, $m(1,t)$ is forced to vanish.  The
orbital of $(1,t)$ is called \emph{orientable} if no such conflict
occurs, which happens if and only if $\chi(tht^{-1}h^{-1}) = 1$ for all
$h \in H \cap t^{-1}Ht$ \cite[Def.~3.3 and Prop.~3.6]{Monrep}; these are the
orbitals whose matrices survive the lifting to the Schur cover.  For an
orientable orbital, the element of $\Cr(\rho_\chi)$ with $m(1,t) = 1$,
supported on that orbital, is a \emph{basis matrix}.

\begin{theorem}[\cite{Monrep}, Thm.~3.4]\label{thm:basis}
The centraliser algebra $\Cr(\rho_\chi)$ has a $\mathbb{C}$-basis
consisting of these basis matrices, one for each orientable orbital.
\end{theorem}

By Maschke's theorem $\Cr(\rho_\chi)$ is semisimple; when $\rho_\chi$ is
multiplicity-free it is commutative of dimension~$r$, the basis matrices
$M_1, \ldots, M_r$ of Theorem~\ref{thm:basis} are simultaneously
diagonalisable with common eigenspaces $V_1, \ldots, V_r$, and the
\textit{character table} $\CT\bigl(\Cr(\rho_\chi)\bigr) =
[\lambda_{i,j}]_{i,j}$ records the eigenvalue $\lambda_{i,j}$ of $M_j$
on $V_i$.

\begin{proposition}[\cite{Monrep}, Prop.~5.1]\label{prop:CT}
Let $\{t_1, \ldots, t_r\}$ be representatives of the $H$-double cosets in $G$, and set $k_i = |H : H \cap t_i^{-1}Ht_i|$. Let $M_{G,H}$ be the submatrix of the character table of $G$ whose rows are indexed by the irreducible constituents of $\rho_\chi$, and let $L$ be the matrix with entries
\[
\ell(t_i, C) = \sum_{h \in H} \delta_C(ht_i)\,\chi(h^{-1}),
\]
where $C$ ranges over conjugacy classes of $G$ and $\delta_C$ is the characteristic function of $C$. Then
\[
\CT\bigl(\Cr(\rho_\chi)\bigr) = \tfrac{1}{|H|}\, M_{G,H}\, L^{\top}\, \diag(k_1, \ldots, k_r).
\]
\end{proposition}

Any $M = \sum_{i=1}^r \alpha_i M_i \in \Cr(\rho_\chi)$ has eigenvalue vector $\CT\bigl(\Cr(\rho_\chi)\bigr)\underline{\alpha}$, and $WW^* = wI_n$ holds if and only if every eigenvalue of $W$ has modulus $\sqrt{w}$; so $M$ is a $\CGW(n,w;k)$ if and only if $\underline{\alpha} \in \bigl(\{0\} \cup \langle\zeta_k\rangle\bigr)^r$ and every entry of $\CT(\Cr(\rho_\chi))\underline{\alpha}$ has modulus $\sqrt{w}$.  The coefficient space is finite, so a complete search is feasible.

\begin{algorithm}\label{alg:search}
\textbf{Search for $\CGW(n,w;k)$ matrices in the centraliser algebra of a rank-$r$ group $G$.}
\begin{enumerate}
\item Compute the Schur cover $\widehat{G}$ of $G$ and the preimage $\widehat{H}$ of a point stabiliser $H \leq G$.
\item For each linear character $\chi$ of $\widehat{H}$, compute the centraliser basis $\{M_1, \ldots, M_r\}$ via Theorem~\ref{thm:basis}, and the character table $\CT(\Cr(\rho_\chi))$ via Proposition~\ref{prop:CT}.
\item For each $\underline{\alpha} \in \bigl(\{0\} \cup \langle\zeta_k\rangle\bigr)^r$, form $W = \sum_{i=1}^r \alpha_i M_i$ and check whether $WW^{\ast}$ is a scalar multiple of $I_n$; if so, record $W$ as a $\CGW$ matrix.
\end{enumerate}
\end{algorithm}

The coefficient space in Step~3 has $(k+1)^r$ elements, manageable for $r \leq 5$; because $\CGW$ entries are drawn from a finite set rather than the continuum of complex units, this complete enumeration replaces the Gr\"{o}bner basis computation of~\cite{Monrep}.

\subsection{Schur covers}\label{sec:schur}

Step~1 of the algorithm is justified by the structure theory of strong
automorphism groups, developed for complex Hadamard matrices in
\cite[Section~6]{Monrep}; we record the adaptation to matrices with zero
entries.  Let $W \in \CGW(n,w;k)$, let
$\mathbb{T} = \{z \in \mathbb{C} : |z| = 1\}$, let
$\Gamma = \{L \in \Mon_n(\mathbb{T}) : LWL^{\ast} = W\}$ -- the
first-component projection of the strong automorphism group of $W$,
computed over $\mathbb{T}$ rather than $\langle \zeta_k \rangle$ -- and
let $\pi \colon \Gamma \rightarrow \mathrm{Sym}_n$ be the projection onto
the underlying permutation group.  Call $W$ \emph{indecomposable} if the
graph on $\{1, \ldots, n\}$ with an edge $\{i,j\}$ whenever
$W_{ij} \neq 0$ is connected.  For indecomposable $W$, the kernel of
$\pi$ consists of scalar matrices: a diagonal matrix
$D = \diag(d_1, \ldots, d_n)$ in $\Gamma$ satisfies
$d_i \overline{d_j} = 1$ whenever $W_{ij} \neq 0$, and connectivity
forces all $d_i$ to be equal.  (For a Hadamard matrix every entry is
nonzero, and this is \cite[Prop.~6.1]{Monrep}.)  The next proposition
shows that the search over the linear characters of a single Schur cover
is exhaustive: every indecomposable $\CGW$ matrix whose unimodular strong
automorphism group projects onto $G$ lies, up to monomial equivalence, in
one of the centraliser algebras searched.  No perfectness hypothesis is
required -- \cite[Prop.~6.2]{Monrep} identifies the determinant-one
subgroup of $\Gamma$ exactly when $G$ and that subgroup are perfect, but
exhaustiveness holds in general -- and any one Schur cover suffices, even
though covers of imperfect groups need not be unique.

\begin{proposition}\label{prop:lifting}
Let $W \in \CGW(n,w;k)$ be indecomposable, let
$\Gamma = \{L \in \Mon_n(\mathbb{T}) : LWL^{\ast} = W\}$, and let
$G \leq \pi(\Gamma)$ be a transitive subgroup with point stabiliser $H$.
Let $\widehat{G}$ be any Schur cover of $G$ and let $\widehat{H}$ be the
full preimage of $H$ in $\widehat{G}$.  Then there is a linear character
$\lambda$ of $\widehat{H}$ such that $W$ is monomially equivalent to a
matrix in $\Cr(\rho_\lambda)$, where
$\rho_\lambda = \lambda \myuparrow_{\widehat{H}}^{\widehat{G}}$.
\end{proposition}

\begin{proof}
Let $\Gamma_G = \pi^{-1}(G) \cap \Gamma$.  Since $W$ is indecomposable, the
kernel of $\pi$ on $\Gamma$ consists of scalar matrices, so assigning to
$g \in G$ the image modulo scalars of any preimage in $\Gamma_G$ is a
well-defined homomorphism
$\bar{\rho} \colon G \rightarrow \mathrm{PGL}_n(\mathbb{C})$, that is, a
projective representation of $G$.  By Schur's lifting theorem, every
projective representation of a finite group lifts to a linear
representation of each of its representation groups, and every stem
extension of $G$ by the full Schur multiplier -- in particular every Schur
cover -- is a representation group~\cite{Karpilovsky}.  Thus there is a
linear representation
$\widehat{\rho} \colon \widehat{G} \rightarrow \mathrm{GL}_n(\mathbb{C})$
such that $\widehat{\rho}(\widehat{G})$ and $\Gamma_G$ have the same image
modulo scalars.  In particular, every $\widehat{\rho}(x)$ is a scalar
multiple of a monomial matrix, hence monomial, and the permutation
underlying $\widehat{\rho}(x)$ is the image of $x$ in $G$.  So
$\widehat{\rho}$ is a transitive monomial representation of $\widehat{G}$
whose point stabiliser is exactly $\widehat{H}$, and by
Proposition~\ref{prop:monrep} it is monomially equivalent to
$\rho_\lambda = \lambda \myuparrow_{\widehat{H}}^{\widehat{G}}$ for some
linear character $\lambda$ of $\widehat{H}$.  Finally, a matrix commutes
with a monomial matrix $L$ if and only if it commutes with every scalar
multiple of $L$, so
$\Cr(\widehat{\rho}(\widehat{G})) = \Cr(\Gamma_G) \supseteq \Cr(\Gamma) \ni W$;
applying the monomial change of basis carrying $\widehat{\rho}$ to
$\rho_\lambda$ yields the claim.
\end{proof}

\subsection{A non-existence criterion}\label{sec:lam-leung}

We derive a necessary condition for the existence of $\CGW$ matrices in a centraliser
algebra, using the following theorem of Lam and Leung.

\begin{theorem}[{\cite[Theorem~1.2]{LamLeung}}]\label{thm:lam-leung}
There exists a vanishing sum $\sum_{i=1}^{\mu} \zeta_k^{a_i} = 0$ of $\mu$ many
$k$-th roots of unity if and only if $\mu$ is expressible as a non-negative integer
combination of the distinct prime factors of $k$.
\end{theorem}

\begin{proposition}\label{prop:lam-leung-cgw}
Let $G$ be a rank-$3$ permutation group on $n$ points, with point stabiliser $H$. Let $\Gamma$ be a strongly regular graph with parameters $(n, \alpha, \lambda, \mu)$ on which $G$ acts. Let $\hat{G}$ be the Schur multiplier of $G$, let $\hat{H}$ be the preimage of $H$ in $\hat{G}$, and let $\chi$ be a linear character of $\hat{H}$. Suppose that $W \in \Cr(\rho_\chi)$ is a $\CGW(n,w;k)$. Then there exists a vanishing sum of $k^{\textrm{th}}$ roots of unity of length $\mu$. 
\end{proposition}

\begin{proof}
For any two non-adjacent rows $i,j$ of $W$, the $(i,j)$-entry of $WW^* = wI_n$ gives $\sum_\ell W_{i\ell}\,\overline{W_{j\ell}} = 0$.
Each nonzero term is a $k$-th root of unity, and the number of such terms equals $\mu$. Since $i, j$ are non-adjacent, the number of non-zero terms in the sum does not depend on whether the diagonal is zero or non-zero, and the result follows. 
\end{proof}

If $k$ is prime, then $\mu$ is divisible by $k$; this yields a number of non-existence results for strongly regular graphs with small parameters: e.g. the Clebsch graph has $\mu = 2$, so it cannot yield a $\CGW(16,5;5)$ or $\CGW(16,5;3)$. A more general argument may be obtained in terms of intersection numbers of the centraliser algebra, considered as an association scheme, but we do not pursue this here.

\section{Results}\label{sec:families}

We applied Algorithm~\ref{alg:search} to the primitive permutation
groups of rank $3$, $4$ and $5$ and degree $\leq 80$. This section records the scope of the computation, and describes some matrices found. The complete classification appears in the tables of the appendix.

The search was implemented independently by two of the authors, in
\textsc{Magma} and in \textsc{SageMath}/\textsc{GAP} (the latter using the
\textsc{AtlasRep} database~\cite{AtlasRep} for the covers of almost-simple
groups). Algorithm~\ref{alg:search} was run on every primitive group of
rank $r \in \{3,4,5\}$ and degree $n \leq 80$ other than $n=64$. At degree $64$ there are groups of affine type for which standard routines in {\sc SAGE} and {\sc Magma} did not return Schur covers. The search at this degree is therefore
incomplete, though a number of matrices were found. Note that a group may act
primitively on $n$ points in more than one way, and the preimage of a
point stabiliser in a Schur cover may comprise several conjugacy classes
of subgroups of index $n$; each such action and each such class is
treated separately, since the matrices they produce need not be
monomially equivalent. 

We run the search for each
$k \in \{2,3,4,5,6\}$; normalising the coefficient of $M_1 = I_n$ to
$\{0,1\}$ by a global scalar leaves $2(k+1)^{r-1}$ candidates per character.
A partial search at larger degrees is described below.  We emphasise that the
bound $k \leq 6$ applies to the \emph{coefficients} of the linear
combination: the entries of the basis matrices $M_i$ are values of $\chi$,
so the entries of a matrix returned by the search lie in
$\langle \zeta_K \rangle$ with $K = \mathrm{lcm}(k, \mathrm{ord}(\chi))$,
and phases as large as $K = 30$ occur. Each matrix in our census is
labelled by its \emph{minimal} phase, found by checking, for every proper
divisor $k'$ of $k$, whether some scalar multiple $xW$ has all entries in
$\{0\} \cup \langle \zeta_{k'} \rangle$.
This computation is exhaustive, so the resulting classification is
complete, with the caveats noted above. Combining the computations with Proposition~\ref{prop:lifting} gives a precise statement of what has been classified.

\begin{theorem}\label{thm:scope}
Let $n \leq 80$ with $n \neq 64$, $r \in \{2,3,4,5\}$ and
$k \in \{2,3,4,5,6\}$.  Let $W$ be an
indecomposable $\CGW$ matrix of order $n$ such that
\begin{enumerate}
\item the unimodular strong automorphism group of $W$ projects onto a
primitive permutation group $G$ of degree $n$ and rank $r$; and
\item some matrix in $\Cr(\rho_\lambda)$ monomially equivalent to $W$, as
furnished by Proposition~\ref{prop:lifting}, has -- after multiplication by
a global scalar -- its coefficient of $M_1 = I_n$ in $\{0,1\}$ and all
remaining orbital coefficients in $\{0\} \cup \langle \zeta_k \rangle$.
\end{enumerate}
Then $W$ is monomially equivalent to a matrix in our census.
\end{theorem}

The complete output is presented in
Tables~\ref{tab:aut_rank2}--\ref{tab:aut_rank4} of the appendix, split by
the rank of the automorphism group: $2$-transitive groups, rank-$3$ groups,
primitive groups of rank at least $4$, and imprimitive groups. Column
headings are defined in the appendix preamble.

\subsection{Weighing matrices from finite geometry}\label{sec:geom}

Investigation of the matrices in the census for potential infinite families uncovered two previously known families coming from finite geometry. The first is supported on a projective plane.

\begin{theorem}[Seberry--Whiteman~\cite{SeberryWhiteman1975}, Berman~\cite{Berman1978}]\label{thm:seberrywhiteman}
For every prime power $q$ and $k$ dividing $q-1$ there exists a $\CGW(q^2+q+1, q^2; k)$
whose automorphism group contains $\mathrm{PSL}(3,q)$ acting $2$-transitively.
\end{theorem}

Our search recovers this family for $q = 2$ and $q = 3$: we find
$\CGW(7,4;k)$ for $k \in \{2,6\}$ with automorphism group
$\mathrm{PSL}(3,2)$, and $\CGW(13,9;k)$ for $k \in \{2,4,6\}$ with
automorphism group $\mathrm{PSL}(3,3)$; in each case the group acts
$2$-transitively and flag-transitively, with a single flag-orbit.
The second family is supported on the Grassmannian of a vector space over
a finite field, and is due to Goldberger.

\begin{theorem}[Goldberger~{\cite[Theorem~6.9]{Goldberger2024}}]\label{thm:goldberger}
Let $q$ be a prime power, let $r > 0$ be integers, and let $k > 1$
divide $q-1$.  Then there exists a
$\CGW\bigl(\genfrac{[}{]}{0pt}{}{2r}{r}_q,\; q^{r^2};\; k\bigr)$,
where $\genfrac{[}{]}{0pt}{}{2r}{r}_q$ is the Gaussian binomial coefficient,
supported on the relation of complementary pairs in the Grassmannian of
$r$-subspaces of $\F_q^{2r}$, whose automorphism group contains
$\mathrm{PGL}(2r,q)$ acting on that Grassmannian.
\end{theorem}

The case $(r,k) = (2,2)$ is supported on the lines of
$\mathrm{PG}(3,q)$, and recovers at $q = 2$ the matrices
$\CGW(35,16;k)$, $k \in \{2,4,6,10\}$, which have monomial automorphism group
$A_8 \cong \mathrm{PSL}(4,2)$ on the $35$ lines.  The
first member beyond $q = 2$, a $\We(130,81)$ at $q = 3$, lies outside the
range $n \le 80$ of our classification.

\subsection{Weighing matrices in Hamming schemes}\label{sec:hamming}

Fix a normal $q\times q$ matrix $C$ with entries in $\{0\}\cup\langle\zeta_k\rangle$,
the \emph{base}.  Write $X_0=I_q$ and $X_1=C$, and for $d\geq1$ form
\[
  M_j=\sum_{\substack{\epsilon\in\{0,1\}^d\\|\epsilon|=j}}\ X_{\epsilon_1}\otimes\cdots\otimes X_{\epsilon_d},
  \qquad 0\leq j\leq d,
\]
where $|\epsilon|=\epsilon_1+\cdots+\epsilon_d$, so that the sum runs over the
$\{0,1\}$-strings with $j$ ones.  The $M_j$ commute, and we diagonalise them
simultaneously.  As $C$ is normal, $\C^q$ has a basis of eigenvectors of $C$;
choosing one such eigenvector $v_l$ in each of the $d$ tensor factors, the
products $v_1\otimes\cdots\otimes v_d$ form a basis of $\C^{q^d}$.

\begin{lemma}\label{lem:basis-eigenvalues}
If $Cv_l=\theta_l v_l$ for $l=1,\dots,d$, then $v_1\otimes\cdots\otimes v_d$ is an
eigenvector of $M_j$ with eigenvalue $e_j(\theta_1,\dots,\theta_d)$, the $j$-th
elementary symmetric polynomial in $\theta_1,\dots,\theta_d$.  As these products
range over the basis above, they give the entire spectrum of $M_j$.
\end{lemma}

\begin{proof}
The summand of $M_j$ indexed by $\epsilon$ scales $v_1\otimes\cdots\otimes v_d$ by
$\prod_{l:\,\epsilon_l=1}\theta_l$; summing over the strings with $|\epsilon|=j$
gives the eigenvalue $\sum_{|\epsilon|=j}\prod_{l:\,\epsilon_l=1}\theta_l
=e_j(\theta_1,\dots,\theta_d)$.
\end{proof}

Taking $C=J_q-I_q$ recovers the classical \emph{Hamming scheme} $H(d,q)$, with
vertex set labelled by $d$-tuples drawn from a set of size $q$ and relations $D_0,\dots,D_d$ in which $(x,y)\in D_j$ when $x$
and $y$ differ in exactly $j$ coordinates; $D_j$ has valency
$v_j=\binom{d}{j}(q-1)^j$, and the automorphism group is $S_q\wr S_d$ in the
product action.  Here $M_j$ is the adjacency matrix $A_j$, and the $A_j$ span the
metric Bose--Mesner algebra, generated by $A_1$.  The base $J_q-I_q$ has
eigenvalues $q-1$, once, and $-1$, with multiplicity $q-1$; by
Lemma~\ref{lem:basis-eigenvalues} the eigenvalue of $A_j$ on the common
eigenspace $V_i$ where $i$ of the coordinates carry $-1$ is the elementary
symmetric polynomial in $d-i$ copies of $q-1$ and $i$ copies of $-1$, that is the
Krawtchouk polynomial
\[
  K_j(i)=\sum_{\ell=0}^{j}(-1)^{\ell}(q-1)^{j-\ell}\binom{i}{\ell}\binom{d-i}{j-\ell},
  \qquad K_j(0)=v_j,
\]
see~\cite{BCN}.

The matrices of the census arise on replacing $J_q-I_q$ by a root-of-unity base.
A linear character $\chi$ of $\widehat H$ -- the preimage in the Schur cover
$\widehat G$ of $S_q\wr S_d$ of a point stabiliser -- determines the centraliser
algebra $\Cr(\rho_\chi)$ (Section~\ref{sec:prelim}): in each coordinate $J_q-I_q$
is replaced by the basis matrix $C$ on $K_q$, with root-of-unity entries fixed by
$\chi$, and $A_j$ by the basis matrix $M_j$ on $D_j$ (the Schur cover of a wreath
product is known from the homology of wreath products~\cite{Karpilovsky}).
Lemma~\ref{lem:basis-eigenvalues} applies verbatim to the eigenvalues of $C$.

The base $C$ has zero diagonal and root-of-unity entries off it, so
$\operatorname{tr}C=0$ and $\operatorname{tr}(CC^*)=q(q-1)$.  The action on the
$q$ points has rank~$2$, so $\Cr(\rho_\chi)$ is spanned by $I_q$ and $C$
(Theorem~\ref{thm:basis}); thus $C$ has exactly two distinct eigenvalues
$\sigma,\tau$, of multiplicities $m,m'$ satisfying $m\sigma+m'\tau=0$ and
$m|\sigma|^2+m'|\tau|^2=q(q-1)$.  We call $C$ a \emph{conference matrix} when
$CC^*=(q-1)I$; this holds if and only if $\tau=-\sigma$, equivalently $m=m'$, and
then $q$ is even and $|\sigma|^2=|\tau|^2=q-1$.  In that case $C^2=\sigma^2I$, and
since the diagonal entry $\sigma^2=\sum_l C_{il}C_{li}$ is a sum of $q-1$
unit-modulus terms of modulus $q-1$, these coincide: $C^2=(q-1)\omega I$ for a
root of unity $\omega$.

A combination of the basis matrices $M_j$ is a weighing matrix exactly when all
of its eigenvalues have a single absolute value; for a single basis matrix this
asks that one column of the character table be of constant absolute value.  Call
a weighing matrix \emph{flag-transitive} if its strong automorphism group is
transitive on the positions of its nonzero entries.  Since $S_q\wr S_d$ acts
transitively on the support of each $M_j$, a matrix is flag-transitive precisely
when it is a single basis matrix.

\begin{theorem}\label{thm:hamming-flag}
Let $d\geq2$, let $\chi$ be a linear character of $\widehat H$, and let $M_j$
($1\leq j\leq d$) be a basis matrix of $\Cr(\rho_\chi)$.  Then $M_j$ is a weighing
matrix if and only if $j=d$ and the base $C$ is a conference matrix; this forces
$q$ even, and then $M_d=C^{\otimes d}\in\CGW(q^d,(q-1)^d;k)$.  In particular, for
$q$ odd the Hamming scheme $H(d,q)$ carries no flag-transitive weighing matrix.
\end{theorem}

\begin{proof}
By Lemma~\ref{lem:basis-eigenvalues}, taking all $d$ coordinates equal to an
eigenvector for $\sigma$ gives an eigenvector of $M_j$ with eigenvalue
$\binom{d}{j}\sigma^{j}$, and likewise $\binom{d}{j}\tau^{j}$ for $\tau$.  The
matrix $M_j$ is supported on $D_j$, so has weight $\binom{d}{j}(q-1)^j$; if it is
a weighing matrix, every eigenvalue has squared modulus equal to this weight, so
$\binom{d}{j}^2|\sigma|^{2j}=\binom{d}{j}^2|\tau|^{2j}$ and hence
$|\sigma|^2=|\tau|^2$.  With $m\sigma+m'\tau=0$ this forces $\tau=-\sigma$, so $C$
is a conference matrix, $q$ is even, and $|\sigma|^2=q-1$.  The weight condition
then reads $\binom{d}{j}^2(q-1)^{j}=\binom{d}{j}(q-1)^{j}$, giving $\binom{d}{j}=1$
and $j=d$.  Now $M_d=C^{\otimes d}$ satisfies
$M_dM_d^{*}=(CC^{*})^{\otimes d}=(q-1)^dI$, so $M_d\in\CGW(q^d,(q-1)^d;k)$;
conversely $C^{\otimes d}$ is such a matrix whenever $C$ is a conference matrix.
For $q$ odd no conference base exists, so $H(d,q)$ carries no flag-transitive 
weighing matrix.
\end{proof}

These Kronecker powers $M_d=C^{\otimes d}$ of conference matrices are a classical construction~\cite{GeramitaSeberry1979}. Such weighing matrices with automorphism groups of rank 3 and 4 are contained in the census tables at degrees 16, 36, 64. 

Theorem~\ref{thm:hamming-flag} treats a single basis matrix; the centraliser
algebra also contains weighing matrices that are \emph{sums} of basis matrices.
For two summands and $d=3$ these are classified.

\begin{theorem}\label{thm:two-relation}
Let $d=3$, and let $M$ be a sum of two basis matrices of $\Cr(\rho_\chi)$. Then $M \in \mathrm{CGW}(q^{3},w;k)$
only if the base $C$ is a conference matrix, hence $q$ even; the matrix $M$ is then
$I+S$, with $S=\beta\,C^{\otimes3}$ anti-Hermitian, for some $\beta \in \langle \zeta_{k} \rangle$, lying in
$\CGW\bigl(q^3,\,1+(q-1)^3;\,k\bigr)$, or it is one of two matrices on the binary
cube $H(3,2)$.
\end{theorem}

\begin{proof}
Scaling by a root of unity, take $M = M_a+\beta M_b$
with $0\le a<b\le3$.  Let $\sigma,\tau$ be the
eigenvalues of $C$ and $s=\tau/\sigma$; the relation $m\sigma+m'\tau=0$ makes
$s=-m/m'$ a negative rational.  By Lemma~\ref{lem:basis-eigenvalues} the common
eigenspaces are $V_0,\dots,V_3$, where $V_i$ is spanned by the tensors with
exactly $i$ factors an eigenvector for $\tau$; on $V_i$ the basis matrix $M_c$
($0\le c\le3$) acts as the scalar $\sigma^{c}Q_c(i)$, where $Q_c(i)$ is $e_c$
evaluated at $i$ copies of $s$ and $3-i$ copies of $1$ (the factor $\sigma^{c}$ is
pulled out by the homogeneity of $e_c$).  The values $Q_c(i)$ for $i=0,1,2,3$ are
\[
  Q_1:\ 3,\ s{+}2,\ 2s{+}1,\ 3s; \qquad
  Q_2:\ 3,\ 2s{+}1,\ s^2{+}2s,\ 3s^2; \qquad
  Q_3:\ 1,\ s,\ s^2,\ s^3.
\]
The eigenvalues of $M_{0}$ are all equal to $1$, so by abuse of notation we let $Q_{0}(i) = 1$ for all $i$. Put $p=\operatorname{Re}(\beta\sigma^{b-a})$ and $g=|\sigma|^{2(b-a)}$, so that
$p^2\le g$ (as $|\beta|=1$).  The eigenvalue of $M_a+\beta M_b$ on $V_i$ is
$\sigma^{a}Q_a(i)+\beta\sigma^{b}Q_b(i)$, of squared modulus
$|\sigma|^{2a}\bigl(Q_a(i)^2+2p\,Q_a(i)Q_b(i)+g\,Q_b(i)^2\bigr)$; hence the sum is
a weighing matrix precisely when
\begin{equation}\label{eq:wm-flat}
  Q_a(i)^2+2p\,Q_a(i)Q_b(i)+g\,Q_b(i)^2\ \text{ is independent of } i.
\end{equation}

First we show that if $M$ is a weighing matrix, we must have $s=-1$, so $C$ is a conference matrix and
$q$ is even.  If $a=0$ then 
\eqref{eq:wm-flat} requires the
quadratic $1+2pX+gX^2$ to take a single value at $X=Q_b(0),\dots,Q_b(3)$.  As
$g=|\sigma|^{2b}>0$ this quadratic is non-constant, so takes each value at most
twice, and the four numbers $Q_b(i)$ occupy at most two values.  For
$b=1$ they form an arithmetic progression of common difference $s-1\neq0$, hence
are distinct and no weighing matrix arises; for $b=2,3$ they pair up only at
$s=-1$.  If instead $1\le a<b$, regard the four instances $i=0,1,2,3$ of
\eqref{eq:wm-flat} as a linear system in $p$, $g$ and the common value $c$; a
solution exists only if
\[
  \det\bigl(Q_a(i)^2,\ Q_a(i)Q_b(i),\ Q_b(i)^2,\ 1\bigr)_{i=0}^{3}=0
\]
as each instance asserts that the row $\bigl(Q_a(i)^2,\,Q_a(i)Q_b(i),\,Q_b(i)^2,\,1\bigr)$
is orthogonal to the fixed nonzero vector $(1,\,2p,\,g,\,-c)$, so the four rows are
linearly dependent.  From the values above this determinant is a nonzero multiple of
$(s{-}1)^6(s{+}1)(s^2{+}4s{+}1)$, of $s(s{-}1)^6(s{+}1)^2(s^2{+}5s{+}1)$, and of
$s^3(s{-}1)^6(s{+}1)^3$ for $(a,b)=(1,2),(1,3),(2,3)$ respectively.  The factors
$s$ and $s-1$ and the quadratics $s^2+4s+1$, $s^2+5s+1$ (of discriminant $12$
and $21$) have no negative rational zero, so a negative rational $s$ annihilates
only the factor $s+1$; hence $s=-1$.  In every
non-vacuous case $s=-1$, i.e.\ $\tau=-\sigma$, so by the conference
characterisation above $C$ is a conference matrix and $q$ is even.

Finally we classify the solutions when $s=-1$.  Here $|\sigma|^2=q-1$, and the values
$Q_c(i)$ for $i=0,1,2,3$ become
\[
  Q_1:\ 3,1,-1,-3; \qquad Q_2:\ 3,-1,-1,3; \qquad Q_3:\ 1,-1,1,-1.
\]
For $\{M_0,M_3\}$ the values $Q_3(i)=\pm1$ make \eqref{eq:wm-flat} hold iff $p=0$,
with $g$ unconstrained: $\beta\sigma^3$ is imaginary, so $S:=\beta\,C^{\otimes3}$
is anti-Hermitian.  By Theorem~\ref{thm:hamming-flag},
$C^{\otimes3}\in\CGW(q^3,(q-1)^3;k)$, so $SS^*=(q-1)^3I$ and
\[
  (I+S)(I+S)^*=I+(S+S^*)+SS^*=\bigl(1+(q-1)^3\bigr)I
\]
for every even $q$ -- the stated family.  The other two-value pairs survive only
at $q=2$.  For $\{M_0,M_2\}$ the values $Q_2(i)\in\{3,-1\}$ give $p=-g$; with
$p^2\le g$ this forces $g\le1$, while $g=|\sigma|^4=(q-1)^2$, so $q=2$.  For
$\{M_1,M_3\}$ one finds $p=-1$ and $g=(q-1)^2$; here
$\beta\sigma^2/(q-1)=\beta\omega$ is a root of unity (using $\sigma^2=(q-1)\omega$)
with real part $p/(q-1)=-1/(q-1)$, and a root of unity with rational real part has
real part in $\{0,\pm\tfrac12,\pm1\}$, so $q-1\in\{1,2\}$ and, $q$ being even,
$q=2$.  The
remaining pairs are impossible: $\{M_1,M_2\}$ forces $g=-1$, and $\{M_2,M_3\}$
forces $p=0$ and $p=-2$ at once.  Hence for $q>2$ only the family $I+S$ occurs,
and at $q=2$ the further pairs $\{M_0,M_2\}$ and $\{M_1,M_3\}$ give two matrices
on the binary cube $H(3,2)$.
\end{proof}

The census realises this family at $q=4$: the matrix $I+M_3\in\CGW(64,28;6)$ on
$H(3,4)$, of weight $1+(q-1)^3=28$, together with the two matrices on the binary
cube $H(3,2)$.

Sums of more basis matrices also occur -- for instance the Butson matrix
$I+M_1+M_2+M_3=(I+C)^{\otimes3}\in\BH(64;6)$ on $H(3,4)$, the sum of all four.

Not every weighing matrix at degree $64$ comes from a Hamming scheme.  Degree
$64=2^6$ also admits primitive affine groups, with elementary abelian translation
group $C_2^6$, whose centraliser algebras contribute many weighing matrices of
their own.  The $\CGW(64,36;6)$ in the census, for instance, arises from such a
group, with sub-orbits of lengths $1,18,18,27$, not the Hamming valencies
$1,9,27,27$.

\subsection{Further examples}\label{sec:further}

Excluding the 2-transitive families of Butson matrices, infinite families containing the matrices in our census are unknown to us. In this section we describe a number of matrices of particular interest. The matrices $\CGW(36,25;k)$,
$k \in \{2,4,6,10\}$, admit the rank-$3$ group $(A_5^2){:}D_4$; the real
class is the Kronecker square $C_6 \otimes C_6$ of the order-$6$ conference
matrix, while the classes at $k \in \{4,6,10\}$ are not real, and so not
scalar multiples of it.  The matrices $\CGW(16,5;k)$ admit the
automorphism group of the Clebsch graph and exist for the six phases
$k \in \{2,4,6,10,20,30\}$.

The simple group $\mathrm{PSL}(3,4)$, in its rank-$3$
action on $56$ points, yields a monomial cover of the Gewirtz graph
$\mathrm{SRG}(56,10,0,2)$: a character of order $4$ of its $48$-fold
Schur cover produces a $\CGW(56,10;4)$, together with
a real $\We(56,45)$ and a $\CGW(56,46;4)$. The two rank-$3$ actions of
$\mathrm{O}_5(3)$ on $40$ points yield the weighing matrices
$\We(40,27)$ and $\We(40,28)$, and the product action of $A_5 \wr S_2$
on $60$ points yields $\We(60,25)$ and $\We(60,36)$; in each of the
last two degrees the census had previously contained no primitive
example. Finally, we record a construction of a weighing matrix based on 
an action of the Hall-Janko group $J_2$. 

\begin{proposition}\label{prop:hj}
Let $\Gamma$ be the Hall--Janko graph $\mathrm{SRG}(100,36,14,12)$.  There
is a signing $S$ of the adjacency matrix of $\Gamma$ with $S^{\top} = -S$
and $SS^{\top} = 36 I_{100}$.  In particular $S \in \We(100,36)$ and
$I_{100} + S \in \We(100,37)$, and both matrices admit the double cover
$2{\cdot}J_2$ as a group of strong automorphisms covering the rank-$3$
action of $J_2$.
\end{proposition}

We are not aware of a previously published
skew-symmetric $\We(100, 36)$ admitting an action of the Hall--Janko group; 
it is not among the structures associated with that graph in Brouwer's
catalogue~\cite{BrouwerHJ}.

Beyond degree $100$ the partial search yields further matrices with
sporadic and classical automorphism groups, collected in
Table~\ref{tab:aut_gt80}.  Among these are weighing and Butson Hadamard
matrices invariant under the Mathieu group $M_{12}$, at degrees $144$,
$220$ and $495$, and under the Held group $\mathrm{He}$, at degree
$2058$; and matrices invariant under $\mathrm{PSL}(4,3)$,
$\mathrm{PSL}(3,4)$ and $\mathrm{PSL}(4,5)$ at degrees up to $806$, the
last of which realises further members of the Grassmannian family of
Theorem~\ref{thm:goldberger} ($\We(156,125)$ at $q=5$, $d=4$, $k=1$, and
$\We(806,625)$ at $q=5$, $d=4$, $k=2$).

\section{Quantum error-correcting codes}\label{sec:quantum}

\subsection{Construction from CGW matrices}

We recall the standard construction of quantum stabiliser codes from Hermitian self-orthogonal codes, in the form applied to CGW matrices in~\cite{CGW-Survey}.
Let $q$ be a prime power with $k = q+1$, and let $\alpha \in \GF(q^2)^{\times}$ be an element
of multiplicative order $k$ (which exists since $k \mid q^2-1$). Define a map
$\varphi \colon \{0\} \cup \langle \zeta_k \rangle \to \GF(q^2)$ by $\varphi(0) = 0$ and
$\varphi(\zeta_k^j) = \alpha^j$. For $W \in \CGW(n,w;k)$, let
$C = \mathrm{rowspace}(\varphi(W))$ denote the $\GF(q^2)$-linear code generated by the rows of
the image matrix.

\begin{theorem}[{\cite[Section~4]{CGW-Survey}}]\label{thm:egan-qecc}
Let $q$ be a prime power, $k = q+1$, and $W \in \CGW(n,w;k)$ with
$\mathrm{char}\,\GF(q) \mid w$. Then $C = \mathrm{rowspace}(\varphi(W))$ is a Hermitian
self-orthogonal $[n,r]_{q^2}$ code, where $r = \dim C$, and yields a quantum stabiliser code
with parameters $[[n, n-2r, \geq d]]_q$, where $d$ is the minimum distance of the Hermitian dual
$C^{\perp_H}$. 
\end{theorem}

\begin{remark}
The minimum distance parameter of the quantum code constructed by the method of Theorem \ref{thm:egan-qecc} is the minimum weight of a codeword in $C^{\perp_{H}} \setminus C$ (except in the case of quantum $[[n,0,d]]$ codes where conventionally the minimum distance parameter is the minimum distance $d$ of $C = C^{\perp_{H}}$) and in some cases it is possible that the true minimum distance of the quantum code is greater than $d$. See \cite[Corollary 19]{Ketkar} and surrounding discussion.  
\end{remark}

The characteristic condition $\mathrm{char}\,\GF(q) \mid w$ ensures that each row of
$\varphi(W)$ is self-orthogonal with respect to the Hermitian inner product
$\langle u, v \rangle_H = \sum_i u_i v_i^q$ over $\GF(q^2)$; orthogonality of distinct rows
follows from the off-diagonal vanishing in $WW^* = wI_n$, since $x^{q} = x^{-1}$ for any $x$ in the image of $\varphi$. The minimum distance satisfies the quantum Singleton bound $d \leq r+1$,
and codes achieving equality are called \emph{quantum MDS}.

\subsection{Computational results}

We apply Theorem~\ref{thm:egan-qecc} to every matrix in our census that satisfies the
characteristic condition. This covers all phases $k \in \{3,4,6,10\}$, corresponding to
$q \in \{2,3,5,9\}$, for which the weight $w$ is divisible by
$\mathrm{char}\,\GF(q)$.  The characteristic condition is restrictive:
for $k = 5$ (so $q = 4$) every matrix in the census has odd weight
($w \in \{5, 25\}$), so no quantum code arises -- the only phase in our
range for which the condition is never satisfied.

Minimum distances are computed exactly via Zimmermann's algorithm (information-set decoding in
\textsc{SageMath}, used for $q \in \{2,3\}$) or via the MacWilliams transform (enumerating all
$|C|=(q^2)^r$ codewords when $|C|\leq 6\times 10^5$). For codes over $\GF(25)$ and $\GF(81)$
direct enumeration is infeasible; here we combine the monomial automorphism group of the CGW
matrix with the Hermitian MacWilliams identity. The group partitions codewords into orbits
(recorded by a BFS over the info-space action), and weight enumeration over orbit representatives
determines the minimum distance exactly.

Two cases require different methods.  The matrix $\CGW(60,4;3)$ is a direct sum of
twelve copies of $\CGW(5,4;3)$, so its code decomposes accordingly, and the dual
distance is that of a single block.  For $\CGW(64,27;10)$ the code $C$ has
$81^8$ codewords, beyond enumeration even by orbits; instead we compute the dual
distance directly.  Since $d(C^{\perp_H})$ equals the minimal number of linearly
dependent columns of the $8 \times 64$ generator matrix of $C$, an exhaustive
search over column subsets of size at most $6$ -- reduced by the monomial
automorphism group, of order $829440$ and transitive on columns -- finds
dependent triples and no dependent pairs, so $d = 3$ exactly.

Table~\ref{tab:qcodes} summarises all quantum codes obtained; every minimum
distance is exact.

\begin{table}[H]
\centering
\small
\caption{Quantum codes constructed from the census.  Each entry gives the
source $\CGW(n,w;k)$, the dimension $r$ of the Hermitian self-orthogonal
code, and the resulting $[[n,K,d]]_q$ code.  All distances are exact;
$*$ marks quantum MDS codes.}
\label{tab:qcodes}
\renewcommand{\arraystretch}{1.15}
\begin{tabular}{lcl@{\qquad}lcl}
\hline
$(n,w;k)$ & $r$ & $[[n,K,d]]_q$ & $(n,w;k)$ & $r$ & $[[n,K,d]]_q$ \\
\hline
$(5,4;3)$   & $2$ & $[[5,1,3]]_2^{*}$   & $(36,36;3)$  & $9$  & $[[36,18,4]]_2$ \\
$(10,4;3)$  & $4$ & $[[10,2,3]]_2$      & $(16,5;6)$   & $8$  & $[[16,0,5]]_5$ \\
$(10,9;4)$  & $3$ & $[[10,4,4]]_3^{*}$  & $(36,25;6)$  & $9$  & $[[36,18,4]]_5$ \\
$(10,9;4)$  & $4$ & $[[10,2,4]]_3$      & $(16,6;4)$   & $8$  & $[[16,0,6]]_3$ \\
$(13,9;4)$  & $3$ & $[[13,7,3]]_3$      & $(16,9;10)$  & $4$  & $[[16,8,3]]_9$ \\
$(13,9;4)$  & $6$ & $[[13,1,4]]_3$      & $(36,36;4)$  & $9$  & $[[36,18,4]]_3$ \\
$(16,9;4)$  & $4$ & $[[16,8,3]]_3$      & $(60,4;3)$   & $24$ & $[[60,12,3]]_2$ \\
$(17,16;3)$ & $4$ & $[[17,9,4]]_2$      & $(64,27;10)$ & $8$  & $[[64,48,3]]_9$ \\
\hline
\end{tabular}
\end{table}

Several entries deserve comment.  The code $[[5,1,3]]_2$ arising from
$\CGW(5,4;3)$ is the celebrated five-qubit perfect quantum
error-correcting code~\cite{CGW-Survey}, the smallest correcting all
single-qubit errors; it is quantum MDS ($d = r+1$), as is $[[10,4,4]]_3$
from $\CGW(10,9;4)$, for which we are not aware of a prior construction.
The $\GF(25)$ codes $[[16,0,5]]_5$ and $[[36,18,4]]_5$, by contrast, fall
short of the best known codes with the same parameters ($[[16,0,8]]_5$ and
$[[36,18,7]]_5$ in~\cite{Grassl:codetables}).  Outside the census proper,
applying Theorem~\ref{thm:egan-qecc} to
$W_5^{\otimes 2} \in \CGW(25,16;3)$, where
$W_5 = \mathrm{circ}(0,1,\zeta_3,\zeta_3,1) \in \CGW(5,4;3)$, gives a
Hermitian self-orthogonal $[25,4]_4$ code and a quantum code
$[[25,17,3]]_2$, its minimum distance computed exactly via the MacWilliams
transform; these parameters match the best known for this length and
dimension~\cite{Grassl:codetables}, there achieved by shortening a stored
$[[40,33,3]]$ code over $\mathrm{GF}(4)$ at sixteen coordinates and
extending by one, whereas the construction from $W_5^{\otimes 2}$ is direct
and algebraic.

The distances of the codes arising from the families $\CGW(4^d,3^d;k)$ of
Theorem~\ref{thm:hamming-flag} are determined for all $d$ by the following
observation.

\begin{proposition}\label{prop:family-distance}
For $d \geq 2$ and $k \in \{4, 10\}$, the quantum codes obtained from
$\CGW(4^d, 3^d; k)$ via Theorem~\ref{thm:egan-qecc} have parameters
$[[\,4^d,\; 4^d - 2^{d+1},\; 3\,]]_{k-1}$.  In particular they are not quantum MDS.
\end{proposition}

\begin{proof}
The rank of $\varphi(W)$ over $\GF((k-1)^2)$ equals $2$ for $W \in \CGW(4,3;k)$
in both cases, and rank is multiplicative under Kronecker products, so
$r = 2^d$ and $K = 4^d - 2^{d+1}$.
If $u \in C_A^{\perp_H}$ for a factor $A$, then for any standard basis vector
$e_j$ and any row $a \otimes b$ of $\varphi(A \otimes B)$,
\[
\langle a \otimes b,\; u \otimes e_j \rangle_H
  \;=\; b_j^{\,q}\,\langle a, u \rangle_H \;=\; 0,
\]
so $u \otimes e_j \in C_{A \otimes B}^{\perp_H}$ with the same weight.  Taking
$A = W^{\otimes 2}$, whose dual contains words of weight $3$
($[[16,8,3]]_3$ and $[[16,8,3]]_9$ in Table~\ref{tab:qcodes}), gives
$d(C^{\perp_H}) \leq 3$ for all $d \geq 2$.  Conversely, column $j$ of
$\varphi(W)$ vanishes exactly in row $j$, so the columns of $\varphi(W)$ have
pairwise distinct supports; columns of a Kronecker product are proportional
only if all their factors are, so no two columns of $\varphi(W^{\otimes d})$
are proportional and no column is zero.  A dual word of weight at most $2$ is
a vanishing combination of at most two columns of the generator matrix, so
none exists, and $d(C^{\perp_H}) = 3$.
\end{proof}

\section{Conclusion}\label{sec:discussion}

We have classified the complex generalised weighing matrices lying in the centraliser algebra of a primitive monomial representation of rank $\leq 5$ and degree $n \leq 80$, with the single exception of degree $64$.  The classification recovers the known geometric and Butson families and, alongside them, produces matrices at parameters for which we found no construction in the literature.

Of the $126$ extended equivalence classes, $35$ are block-diagonal direct
sums and $26$ belong to previously known families -- Butson matrices previously classified by Moorhouse, and weighing matrices classified by Seberry-Whiteman, Berman and Goldberger. A further $16$ arise from the
constructions of Section~\ref{sec:families}. The Butson
classes $\BH(21;3)$~\cite{LOS2020} and the $C_7 \times C_7$-invariant
$\BH(49;6)$~\cite{Szollosi2012, SchmidtWongXiang2021} are known from prior classifications, but to our knowledge the remaining matrices have not been previously constructed. The census also realises a symmetric $\We(21,9)$ invariant
under $\mathrm{PSL}(3,2){:}C_2$, whose existence was open in the Handbook of
Combinatorial Designs~\cite{Handbook2007} and discovered independently by Rosin, ~\cite{Rosin2025}.  The census likewise resolves two open cases in~\cite{CGW-Survey}: $\CGW(15,7;3)$ (invariant under $A_7$, settling the
$\mathrm{SBIBD}(15,7,3)$ lifting problem) and $\CGW(10,7;6)$ (invariant
under $S_5$).

\subsection{Open problems}

Two problems arising from this work seem to us particularly worthwhile.

\begin{enumerate}
\item \textbf{Sums of basis matrices.}\label{prob:hamming-general}
Theorem~\ref{thm:hamming-flag} classifies the weighing matrices that are a
single basis matrix of the centraliser algebra of a Hamming scheme. The algebra
also contains weighing matrices that are sums of several basis matrices -- for
instance the $\BH(64;6)$ on $H(3,4)$. For what values of $d,q,w,k$ does $H(d,q)$ suppose a $\CGW(q^d, w,k)$?

\item \textbf{Conference matrices at odd orders.}\label{prob:conference-seeds}
The families of Theorem~\ref{thm:hamming-flag} have as a base a $\CGW(q,q-1;k)$; examples are known at $q=5$ and $q=9$. For which other odd orders $q$, and which $k$, does a conference matrix exist?
\end{enumerate}

\section*{Acknowledgements}

Andrea \v Svob was supported by Croatian Science Foundation under the project HRZZ-IP-2022-10-4571 and by European Union-NextGenerationEU, project number uniri-iz-25-46-KonGeoGraGru.

This work began at the First Dublin Discrete Mathematics Workshop, organised by Drs Eimear Byrne, Ronan Egan, Pádraig Ó Catháin \& John Sheekey. Funding from Dublin City University Faculty of Humanities and Social Sciences, and from University College Dublin is gratefully acknowledged.

\section*{Data availability}

The census of matrices underlying the classification, together with the
code used to generate and verify it, is available from the authors on
request.

\bibliographystyle{abbrv}
\flushleft{
\bibliography{Biblio}

}

\clearpage
\appendix
\appendix
\section{Classification tables}\label{app:tables}

Columns: $n$ = order, $w$ = weight, $k$ = phase, $\mathrm{cl}$ = number of
classical equivalence classes, $|\mathrm{Aut}|$ = order of the full monomial
automorphism group, $G$ = structure description of the underlying permutation
group, $\mathrm{f}$ = number of flag-orbits.  Structure descriptions follow
\textsc{Atlas} conventions: $A{\times}B$ denotes a direct product, $A{:}B$ a
split extension, and $A{\cdot}B$ a non-split extension.  Ten pairs of classical
classes merge under extended equivalence (Section~\ref{sec:equiv}); the
extended counts, the permutation group orders, and the ranks are omitted
here and are available with the census data.

\begin{center}\scriptsize\setlength{\tabcolsep}{2.5pt}\renewcommand{\arraystretch}{0.97}
\begin{tabular}{rrrrrlr@{\hspace{3.5em}}rrrrrlr}
\toprule
$n$ & $w$ & $k$ & cl & $|\mathrm{Aut}|$ & $G$ & f & $n$ & $w$ & $k$ & cl & $|\mathrm{Aut}|$ & $G$ & f \\
\midrule
$5$ & $4$ & $3$ & $1$ & $180$ & $A_{5}$ & $1$ & $16$ & $6$ & $20$ & $1$ & $115200$ & $C_2^4{:}A_6$ & $1$  \\
$5$ & $4$ & $6$ & $1$ & $360$ & $A_{5}$ & $1$ & $16$ & $10$ & $4$ & $2$ & $7680$ & $\mathrm{ASL}(2,4){:}C_2$ & $1$  \\
$5$ & $5$ & $5$ & $1$ & $500$ & $\mathrm{AGL}(1,5)$ & $1$ & $16$ & $16$ & $2$ & $1$ & $10321920$ & $\mathrm{AGL}(4,2)$ & $1$  \\
$5$ & $5$ & $10$ & $1$ & $1000$ & $\mathrm{AGL}(1,5)$ & $1$ & $16$ & $16$ & $4$ & $1$ & $3840$ & $\mathrm{ASL}(2,4)$ & $2$  \\
$7$ & $4$ & $2$ & $1$ & $336$ & $\mathrm{PSL}(3,2)$ & $1$ & $16$ & $16$ & $4$ & $1$ & $20643840$ & $\mathrm{AGL}(4,2)$ & $1$  \\
$7$ & $4$ & $6$ & $1$ & $1008$ & $\mathrm{PSL}(3,2)$ & $1$ & $16$ & $16$ & $6$ & $2$ & $5760$ & $\mathrm{ASL}(2,4)$ & $2$  \\
$9$ & $9$ & $3$ & $1$ & $11664$ & $\mathrm{AGL}(2,3)$ & $1$ & $16$ & $16$ & $6$ & $1$ & $30965760$ & $\mathrm{AGL}(4,2)$ & $1$  \\
$9$ & $9$ & $6$ & $1$ & $23328$ & $\mathrm{AGL}(2,3)$ & $1$ & $16$ & $16$ & $10$ & $1$ & $51609600$ & $\mathrm{AGL}(4,2)$ & $1$  \\
$10$ & $9$ & $2$ & $1$ & $2880$ & $\mathrm{P\Gamma L}(2,9)$ & $1$ & $17$ & $16$ & $3$ & $1$ & $24480$ & $\mathrm{PSL}(2,16){:}C_2$ & $1$  \\
$10$ & $9$ & $4$ & $1$ & $5760$ & $\mathrm{P\Gamma L}(2,9)$ & $1$ & $17$ & $16$ & $6$ & $1$ & $48960$ & $\mathrm{PSL}(2,16){:}C_2$ & $1$  \\
$10$ & $9$ & $4$ & $1$ & $2880$ & $\mathrm{PGL}(2,9)$ & $1$ & $25$ & $25$ & $5$ & $1$ & $1500000$ & $\mathrm{AGL}(2,5)$ & $1$  \\
$10$ & $9$ & $6$ & $1$ & $8640$ & $\mathrm{P\Gamma L}(2,9)$ & $1$ & $25$ & $25$ & $10$ & $1$ & $3000000$ & $\mathrm{AGL}(2,5)$ & $1$  \\
$10$ & $10$ & $4$ & $1$ & $2880$ & $S_{6}$ & $2$ & $36$ & $36$ & $2$ & $1$ & $2903040$ & $\mathrm{O}(7,2)$ & $2$  \\
$13$ & $9$ & $2$ & $1$ & $11232$ & $\mathrm{PSL}(3,3)$ & $1$ & $40$ & $27$ & $2$ & $1$ & $24261120$ & $\mathrm{PSL}(4,3){:}C_{2}$ & $1$  \\
$13$ & $9$ & $4$ & $1$ & $22464$ & $\mathrm{PSL}(3,3)$ & $1$ & $64$ & $64$ & $2$ & $1$ & $1.65e+14$ & $\mathrm{AGL}(6,2)$ & $1$  \\
$13$ & $9$ & $6$ & $1$ & $33696$ & $\mathrm{PSL}(3,3)$ & $1$ & $64$ & $64$ & $4$ & $1$ & $3.30e+14$ & $\mathrm{AGL}(6,2)$ & $1$  \\
$15$ & $7$ & $3$ & $2$ & $7560$ & $A_{7}$ & $1$ & $64$ & $64$ & $6$ & $1$ & $4.95e+14$ & $\mathrm{AGL}(6,2)$ & $1$  \\
$16$ & $6$ & $4$ & $1$ & $23040$ & $C_2^4{:}A_6$ & $1$  \\
\bottomrule
\end{tabular}
\captionof{table}{CGW matrices with $2$-transitive automorphism group}\label{tab:aut_rank2}
\end{center}
\vspace{2ex}

\begin{center}\scriptsize\setlength{\tabcolsep}{2.5pt}\renewcommand{\arraystretch}{0.97}
\begin{tabular}{rrrrrlr@{\hspace{3.5em}}rrrrrlr}
\toprule
$n$ & $w$ & $k$ & cl & $|\mathrm{Aut}|$ & $G$ & f & $n$ & $w$ & $k$ & cl & $|\mathrm{Aut}|$ & $G$ & f \\
\midrule
$9$ & $8$ & $12$ & $1$ & $432$ & $(C_3^2){:}C_4$ & $2$ & $35$ & $16$ & $4$ & $1$ & $80640$ & $A_{8}$ & $1$  \\
$10$ & $7$ & $6$ & $2$ & $720$ & $S_{5}$ & $2$ & $35$ & $16$ & $6$ & $1$ & $120960$ & $A_{8}$ & $1$  \\
$15$ & $9$ & $2$ & $1$ & $1440$ & $S_{6}$ & $2$ & $35$ & $16$ & $10$ & $1$ & $201600$ & $A_{8}$ & $1$  \\
$16$ & $5$ & $2$ & $1$ & $3840$ & $C_2^4{:}S_5$ & $1$ & $36$ & $25$ & $2$ & $1$ & $57600$ & $S_5 \wr S_2$ & $1$  \\
$16$ & $5$ & $4$ & $1$ & $7680$ & $C_2^4{:}S_5$ & $1$ & $36$ & $25$ & $4$ & $1$ & $115200$ & $S_5 \wr S_2$ & $1$  \\
$16$ & $5$ & $6$ & $1$ & $11520$ & $C_2^4{:}S_5$ & $1$ & $36$ & $25$ & $6$ & $1$ & $172800$ & $S_5 \wr S_2$ & $1$  \\
$16$ & $5$ & $10$ & $1$ & $19200$ & $C_2^4{:}S_5$ & $1$ & $36$ & $25$ & $10$ & $1$ & $288000$ & $S_5 \wr S_2$ & $1$  \\
$16$ & $5$ & $20$ & $1$ & $38400$ & $C_2^4{:}S_5$ & $1$ & $36$ & $26$ & $4$ & $1$ & $57600$ & $A_5^2{:}C_2^2$ & $2$  \\
$16$ & $5$ & $30$ & $1$ & $57600$ & $C_2^4{:}S_5$ & $1$ & $36$ & $36$ & $3$ & $1$ & $777600$ & $A_6 \wr S_2$ & $1$  \\
$16$ & $9$ & $2$ & $1$ & $2304$ & $S_4 \wr S_2$ & $1$ & $36$ & $36$ & $4$ & $1$ & $28800$ & $A_5 \wr S_2$ & $3$  \\
$16$ & $9$ & $4$ & $1$ & $4608$ & $S_4 \wr S_2$ & $1$ & $36$ & $36$ & $6$ & $1$ & $1555200$ & $A_6 \wr S_2$ & $1$  \\
$16$ & $9$ & $6$ & $1$ & $6912$ & $S_4 \wr S_2$ & $1$ & $40$ & $28$ & $2$ & $1$ & $51840$ & $\mathrm{O}(5,3)$ & $2$  \\
$16$ & $9$ & $6$ & $1$ & $6912$ & $C_2^4{\cdot}S_3^2$ & $1$ & $56$ & $10$ & $4$ & $1$ & $161280$ & $\mathrm{PSL}(3,4){:}C_{2}$ & $1$  \\
$16$ & $9$ & $10$ & $1$ & $11520$ & $S_4 \wr S_2$ & $1$ & $56$ & $45$ & $2$ & $1$ & $80640$ & $\mathrm{PSL}(3,4){:}C_{2}$ & $1$  \\
$16$ & $9$ & $12$ & $1$ & $13824$ & $S_4 \wr S_2$ & $1$ & $56$ & $46$ & $4$ & $1$ & $80640$ & $\mathrm{PSL}(3,4)$ & $2$  \\
$16$ & $9$ & $12$ & $1$ & $13824$ & $C_2^4{\cdot}S_3^2$ & $1$ & $64$ & $8$ & $2$ & $1$ & $2580480$ & $C_2^6{:}A_8$ & $1$  \\
$16$ & $9$ & $30$ & $1$ & $34560$ & $S_4 \wr S_2$ & $1$ & $64$ & $14$ & $4$ & $1$ & $451584$ & $(\mathrm{PSL}(3,2)^2){:}C_4$ & $1$  \\
$16$ & $10$ & $4$ & $1$ & $2304$ & $C_2^4{\cdot}S_3^2$ & $2$ & $64$ & $49$ & $2$ & $1$ & $451584$ & $(\mathrm{PSL}(3,2)^2){:}D_4$ & $1$  \\
$16$ & $10$ & $12$ & $1$ & $6912$ & $C_2^4{\cdot}S_3^2$ & $2$ & $64$ & $50$ & $4$ & $1$ & $451584$ & $(\mathrm{PSL}(3,2)^2){:}C_2^2$ & $2$  \\
$35$ & $16$ & $2$ & $1$ & $40320$ & $A_{8}$ & $1$  \\
\bottomrule
\end{tabular}
\captionof{table}{CGW matrices with rank-$3$ automorphism group}\label{tab:aut_rank3}
\end{center}
\vspace{2ex}

\begin{center}\scriptsize\setlength{\tabcolsep}{2.5pt}\renewcommand{\arraystretch}{0.97}
\begin{tabular}{rrrrrlr@{\hspace{3.5em}}rrrrrlr}
\toprule
$n$ & $w$ & $k$ & cl & $|\mathrm{Aut}|$ & $G$ & f & $n$ & $w$ & $k$ & cl & $|\mathrm{Aut}|$ & $G$ & f \\
\midrule
$13$ & $9$ & $2$ & $1$ & $78$ & $C_{13}{:}C_3$ & $3$ & $60$ & $25$ & $2$ & $1$ & $28800$ & $A_5^2{:}C_2^2$ & $2$  \\
$13$ & $9$ & $4$ & $1$ & $156$ & $C_{13}{:}C_3$ & $3$ & $60$ & $36$ & $2$ & $1$ & $28800$ & $A_5^2{:}C_2^2$ & $3$  \\
$13$ & $9$ & $6$ & $1$ & $234$ & $C_{13}{:}C_3$ & $3$ & $64$ & $7$ & $2$ & $1$ & $645120$ & $C_2^6{:}S_7$ & $1$  \\
$16$ & $11$ & $2$ & $1$ & $160$ & $C_2^4{:}C_5$ & $3$ & $64$ & $7$ & $6$ & $1$ & $1935360$ & $C_2^6{:}S_7$ & $1$  \\
$16$ & $11$ & $10$ & $1$ & $800$ & $C_2^4{:}C_5$ & $3$ & $64$ & $7$ & $10$ & $1$ & $3225600$ & $C_2^6{:}S_7$ & $1$  \\
$21$ & $9$ & $2$ & $1$ & $672$ & $\mathrm{PGL}(2,7)$ & $2$ & $64$ & $27$ & $2$ & $1$ & $165888$ & $S_4 \wr S_3$ & $1$  \\
$21$ & $13$ & $6$ & $2$ & $2016$ & $\mathrm{PGL}(2,7)$ & $3$ & $64$ & $27$ & $6$ & $1$ & $497664$ & $S_4 \wr S_3$ & $1$  \\
$21$ & $21$ & $3$ & $2$ & $1008$ & $\mathrm{PGL}(2,7)$ & $4$ & $64$ & $27$ & $10$ & $1$ & $829440$ & $S_4 \wr S_3$ & $1$  \\
$21$ & $21$ & $6$ & $2$ & $2016$ & $\mathrm{PGL}(2,7)$ & $4$ & $64$ & $27$ & $30$ & $1$ & $2488320$ & $S_4 \wr S_3$ & $1$  \\
$35$ & $13$ & $6$ & $2$ & $30240$ & $S_{7}$ & $2$ & $64$ & $28$ & $2$ & $1$ & $82944$ & $S_4^3{:}C_3$ & $2$  \\
$36$ & $20$ & $6$ & $1$ & $4320$ & $\mathrm{PGL}(2,9)$ & $2$ & $64$ & $28$ & $6$ & $1$ & $248832$ & $S_4^3{:}C_3$ & $2$  \\
$36$ & $21$ & $6$ & $2$ & $8640$ & $\mathrm{P\Gamma L}(2,9)$ & $2$ & $64$ & $36$ & $6$ & $1$ & $82944$ & $((((C_2^6{:}C_3^2){:}C_2){:}C_3){:}C_2){:}C_2$ & $2$  \\
$36$ & $21$ & $12$ & $2$ & $17280$ & $\mathrm{P\Gamma L}(2,9)$ & $2$ & $64$ & $36$ & $6$ & $1$ & $82944$ & $(((C_2^6{:}(C_3^2{:}C_3)){:}C_2){:}C_2){:}C_2$ & $2$  \\
$36$ & $21$ & $12$ & $1$ & $17280$ & $\mathrm{P\Gamma L}(2,9)$ & $2$ & $64$ & $56$ & $2$ & $1$ & $7225344$ & $(\mathrm{PSL}(3,2){:}C_2){\times}(C_2^3{:}\mathrm{PSL}(3,2))$ & $1$  \\
$36$ & $30$ & $6$ & $1$ & $4320$ & $\mathrm{PGL}(2,9)$ & $3$ & $64$ & $64$ & $6$ & $1$ & $82944$ & $((((C_2^6{:}C_3^2){:}C_2){:}C_3){:}C_2){:}C_2$ & $4$  \\
$49$ & $49$ & $6$ & $1$ & $10584$ & $(C_7^2){:}(C_6{\times}S_3)$ & $4$ & $65$ & $25$ & $2$ & $1$ & $31200$ & $\mathrm{P\Sigma L}(2,25)$ & $2$  \\
\bottomrule
\end{tabular}
\captionof{table}{CGW matrices with automorphism group of rank $\geq 4$}\label{tab:aut_rank4}
\end{center}

\vspace{2ex}

\begin{center}\scriptsize\setlength{\tabcolsep}{2.5pt}\renewcommand{\arraystretch}{0.97}
\begin{tabular}{rrrrrlr}
\toprule
$n$ & $w$ & $k$ & cl & $|\mathrm{Aut}|$ & $G$ & f \\
\midrule
$100$ & $36$ & $2$ & $1$ & $2419200$ & $J_2{:}C_2$ & $1$  \\
$100$ & $37$ & $2$ & $1$ & $1209600$ & $J_2$ & $2$  \\
$117$ & $81$ & $2$ & $1$ & $24261120$ & $\mathrm{PSL}(4,3){:}C_{2}$ & $2$  \\
$120$ & $21$ & $6$ & $2$ & $241920$ & $\mathrm{PSL}(3,4){:}C_{2}$ & $1$  \\
$120$ & $22$ & $12$ & $2$ & $241920$ & $\mathrm{PSL}(3,4)$ & $2$  \\
$120$ & $57$ & $6$ & $4$ & $483840$ & $\mathrm{PSL}(3,4){:}C_2^2$ & $2$  \\
$120$ & $57$ & $12$ & $2$ & $483840$ & $\mathrm{PSL}(3,4){:}C_{2}$ & $2$  \\
$120$ & $64$ & $6$ & $1$ & $120960$ & $\mathrm{PSL}(3,4)$ & $3$  \\
$130$ & $81$ & $2$ & $1$ & $48522240$ & $\mathrm{PSL}(4,3){:}C_2^2$ & $1$  \\
$130$ & $82$ & $4$ & $1$ & $48522240$ & $\mathrm{PSL}(4,3){:}C_{2}$ & $2$  \\
$144$ & $144$ & $2$ & $1$ & $380160$ & $M_{12}{:}C_{2}$ & $4$  \\
$156$ & $125$ & $2$ & $1$ & $5.80\mathrm{e}{+}10$ & $\mathrm{PSL}(4,5){:}C_{4}$ & $1$  \\
$156$ & $126$ & $4$ & $1$ & $18720000$ & $\mathrm{O}(5,5)$ & $2$  \\
$175$ & $49$ & $2$ & $1$ & $235200$ & $\mathrm{PSL}(2,49){:}C_{2}$ & $2$  \\
$220$ & $72$ & $2$ & $1$ & $190080$ & $M_{12}$ & $1$  \\
$220$ & $73$ & $2$ & $2$ & $190080$ & $M_{12}$ & $2$  \\
$495$ & $273$ & $6$ & $2$ & $1140480$ & $M_{12}{:}C_{2}$ & $4$  \\
$806$ & $625$ & $2$ & $1$ & $1.16\mathrm{e}{+}11$ & $\mathrm{PSL}(4,5){:}D_{4}$ & $1$  \\
$806$ & $626$ & $4$ & $1$ & $1.16\mathrm{e}{+}11$ & $\mathrm{PSL}(4,5){:}C_2^2$ & $2$  \\
$2058$ & $425$ & $2$ & $1$ & $-$ & $\mathrm{He}$ & $-$  \\
$2058$ & $426$ & $4$ & $1$ & $-$ & $\mathrm{He}$ & $-$  \\
\bottomrule
\end{tabular}
\captionof{table}{Matrices of degree $n > 80$ from the partial search; Paley pairs $\We(n,n-1)$, $\mathrm{BH}(n;2)$ from $\mathrm{PSL}(2,q)$ also occur at each $n = q+1 \leq 258$.}\label{tab:aut_gt80}
\end{center}

\vspace{2ex}

\end{document}